\documentclass[pdftex,11pt,svgnames]{article}

\usepackage[utf8]{inputenc}            
\usepackage[T1]{fontenc}
\usepackage{amsmath}
\usepackage{amssymb}  
\usepackage{amsbsy}
\usepackage{listings}
\usepackage{graphics}

\usepackage{tabularx}

\usepackage{pifont}% http://ctan.org/pkg/pifont

\usepackage{floatrow}

\usepackage{enumerate}

\usepackage{bbm}

\usepackage[multiple]{footmisc}

\IfFileExists{dsfont.sty}{
	\usepackage{dsfont}
	\usepackage[usenames,dvipsnames]{color}
	
}{
	
}

\usepackage[numbers]{natbib}

\usepackage{pdfpages}
\usepackage{hyperref}
\hypersetup{
    bookmarks=true,         % show bookmarks bar?
    unicode=false,          % non-Latin characters in Acrobat’s bookmarks
    pdftoolbar=true,        % show Acrobat’s toolbar?
    pdfmenubar=true,        % show Acrobat’s menu?
    pdffitwindow=false,     % window fit to page when opened
    pdfstartview={FitH},    % fits the width of the page to the window
    pdftitle={Forward transition rates for life and pension insurance},    % title
    pdfauthor={Christian Furrer},     % author
    pdfsubject={Actuarial Mathematics},   % subject of the document
    pdfcreator={ },   % creator of the document
    pdfproducer={ }, % producer of the document
    pdfkeywords={ } { } { }, % list of keywords
    pdfnewwindow=true,      % links in new window
    colorlinks=true,       % false: boxed links; true: colored links
    linkcolor=Black, %WildStrawberry,          % color of internal links
    citecolor=Black,        % color of links to bibliography
    filecolor=Black,      % color of file links
    urlcolor=Black %WildStrawberry   % color of external links
}
\usepackage[font=small,format=plain,labelfont=bf,up,textfont=it,up]{caption}

\usepackage[a4paper]{geometry}  % Geometri-pakke: Styrer bl.a. maginer    %
\usepackage[english]{babel}

\usepackage[autolanguage,np]{numprint}

\usepackage{fancyhdr}
\usepackage{amsthm}
\theoremstyle{plain}
\newtheorem{Theorem}{Theorem}
\newtheorem{theorem}[Theorem]{Theorem}

\newtheorem{lemma}[Theorem]{Lemma}

\newtheorem{definition}[Theorem]{Definition}

\theoremstyle{definition}

\theoremstyle{remark}

\numberwithin{Theorem}{section}

\numberwithin{equation}{section}
\makeatletter

\newcommand{\Rmnum}[1]{\expandafter\@slowromancap\romannumeral #1@}
\makeatother

\usepackage{wasysym} % Nu får du cirkler i slutningen af dine exempler og bemærkninger...

\usepackage{tikz}
\usetikzlibrary{arrows,positioning,calc,fit,shapes.geometric,shapes.misc,matrix} 
\tikzset{
    >=stealth',
    punkt/.style={
           rectangle,
           rounded corners,
           draw=black, thick,
           text width=6em,
           minimum height=2.5em,
           text centered},
    punktl/.style={
           rectangle,
           rounded corners,
           draw=black, thick,
           text width=8em,
           minimum height=3em,
           text centered},
    pil/.style={
           ->,
           shorten <=4pt,
           shorten >=4pt,},
    pildotted/.style={
           ->,
           shorten <=4pt,
           shorten >=4pt,
  dotted,},
      pildashed/.style={
           ->,
           shorten <=4pt,
           shorten >=4pt,
  dashed,
  }
}
\tikzset{state/.style={rectangle,rounded corners,draw=black, thick,minimum height=2em,inner sep=0pt,text centered,},}
\usepackage{authblk}

\newcommand{\bigCI}{\mathrel{\text{\scalebox{1.07}{$\perp\mkern-10mu\perp$}}}}

\newcommand{\cmark}{\ding{51}}%

\usepackage{textcomp}

\usepackage{doi}

\usepackage{etoolbox}
\title{Forward transition rates}
\author[1]{Kristian Buchardt}
\author[1,2,3]{Christian Furrer}
\author[2]{Mogens Steffensen}
\affil[1]{\footnotesize PFA Pension, Sundkrogsgade 4, DK-2100 Copenhagen \O, Denmark.}
\affil[2]{\footnotesize Department of Mathematical Sciences, University of Copenhagen, Universitetsparken 5, DK-2100 Copenhagen \O, Denmark.}
\affil[3]{\footnotesize Corresponding author. E-mail: \href{mailto:furrer@math.ku.dk}{furrer@math.ku.dk}.}
\date{}
\begin{document}
\maketitle

\addtocounter{footnote}{3} % Forste fodnote er brugt i author...

\begin{center}
{\sc Abstract}
\end{center}
{\small

The idea of forward rates stems from interest rate theory. It has natural connotations to transition rates in multi-state models. The generalization from the forward mortality rate in a survival model to multi-state models is non-trivial and several definitions have been proposed. We establish a theoretical framework for the discussion of forward rates. Furthermore, we provide a novel definition with its own logic and merits and compare it with the proposals in the literature. The definition turns the Kolmogorov forward equations inside out by interchanging the transition probabilities with the transition intensities as the object to be calculated. \\




%\textbf{2010 Mathematics Subject Classification:} Primary: 62P05. Secondary: 60J28;62F15;62N02.

}

\section{Introduction}

We provide a novel concept of forward transition rates in multi-state models with applications to life insurance as well as credit risk. It is a purely probabilistic concept that is tailor-made to match transition probabilities in a specific way, even in state models that are not Markovian. Though simple and constructive, our forward transition rates are different from the ones suggested in the literature, mainly with applications to life insurance in mind. Our contribution is three-fold. We propose a novel multi-state definition, we analyze its characteristics, and we compare these characteristics with those of other definitions proposed earlier.

Forward interest rates play an important role in bond market theory. They allow us to represent, at a fixed time point, both prices of nominal payments and prices of future interest rates `as if' the future interest rates were known and equal to the forward rates. The forward interest rate curve has even been considered as the fundamental object to model, rather than the interest rate, by a stochastic infinite-dimensional process with certain consistency constraints.

Two areas of finance and insurance that are closely linked, at least from a probabilistic point of view, are (reduced form) credit risk theory and life insurance mathematics. The relation between the areas has been explored and exploited by e.g.\ \citet{KraftSteffensen2007}. In both disciplines, the doubly-stochastic finite-state Markov chain is a fundamental stochastic model. The health and life status of an insured (or, in credit risk theory, the credit rating and solvency status of a firm) is modeled as a finite-state chain. In the doubly-stochastic Markov setting, this finite-state chain is assumed to be Markov, conditional on the transition rates. These rates depend on macro-demographic conditions in the population (or, in credit risk theory, macro-economic conditions in the market and the political regime).

When studying transition probabilities and transition densities in these models, the relation to forward rates in bond market theory is striking -- particularly in the simple survival model, where there are only two states of which one is absorbing. This was first observed and exploited by \citet{MilevskyPromislow2001}, while \citet{MiltersenPersson} discussed the extension to stochastic interest rates, allowing for correlation between mortality and interest. The idea about generalization to multi-state models was discussed and researched by several academics during those years but were put on halt by  \citet{Norberg2010} who explained thoroughly the drawback of each and every natural generalizing definition. As it turned out, this was not enough to kill the idea itself. Lately, \citet{ChristiansenNiemeyer} and \citet{Buchardt2017} have proposed different generalizations with individual characteristics.

We propose here yet another generalizing definition with its own logic and merits. It is based on the simple idea to consider the system of Kolmogorov forward equations not as a means of calculating transition probabilities for given transition rates, but instead as a means of calculating transition rates for given transition probabilities. One version of the idea was already considered by \citet{Norberg2010} but rejected due to general non-uniqueness of the solution, i.e. non-uniqueness of the forward transition rates. But Norberg's version took the initial state for given. If the equations have to hold for a portfolio of insured (or a portfolio of firms in the credit risk version), distributed over the state space at the starting time point, we obtain far stronger results regarding existence and uniqueness.

The definition has drawbacks in specific applications to insurance, though. The transition probabilities arising from our forward transition rates in an `assumed to be' Markovian setting are actually the correct transition probabilities -- this is exactly how they are constructed. However, the transition probabilities and our forward transition rates do not form together, in an `assumed to be' Markovian setting, the densities of transitions. This limits their application for calculation of relevant actuarial quantities. We indicate, however, how this drawback partly can be made up for by extending the model artificially. 

One part of our contribution is our specific proposal. Another part of our contribution is the establishment of a theoretical framework for the discussion and comparison of forward rate definitions. This allows us to give a clear presentation of the relation between the different suggestions pushed forward by \citet{ChristiansenNiemeyer}, \citet{Buchardt2017}, and this paper, and a highlight of the pros and cons of each idea. Our conclusion is different from the negative of \citet{Norberg2010}. We believe that the whole idea of forward transition rates in multi-state models is relevant to actuarial practice, and we provide a substantiating example. Which version of the forward transition rates you should use depends heavily on what you want to use it for. This in itself does not diminuate the power of the concept but it exposes the demand for a thoughtful analysis of it. This is what we provide here.

As mentioned, one area of application is life insurance where finite-state models are generally accepted as the fundamental tool for representation of payment streams and their expected (present) values. However, the idea of forward transition rates is also potentially applicable in credit risk theory. Many credit derivatives specify nominal payments upon transition of a firm’s state of financial health to a different state of financial health. Other derivatives specify nominal payments if a firm’s financial health is in a specific state in the future. These are exactly the payments also evaluated in life insurance, see also \citet{KraftSteffensen2007}.

A key difference between life insurance and credit risk theory is that for life insurance, the transition rates can often reasonably well be assumed to be independent of the interest rate. Uncertainty of a wide range of transition rates is mainly driven by socio-demographic developments which are presumably not, at least not by first order, linked to the economy as such. Thereby, the difficulties arising from correlation between interest and mortality rates in a survival model, pointed out by \citet{MiltersenPersson}, are of little relevance within that domain. In credit risk theory, the uncertainty of the transition rates is mainly driven by socio-economic developments that are, in contrast, strongly linked to the development of the interest rates.

In this exposition, we pay only little attention to the interest rate. Our focus is not on handling correlation with the interest rate but on handling, in the case of no correlation with the interest rate, the challenges arising from generalizing from the survival model to multi-state models. Therefore, it is targeted users of multi-state models in general and those in the area of life insurance in particular. The generalization to include correlation with interest rates is outside the scope here and is, together with discussing the dynamics of forward transition rate curves, postponed to future research.

The article is structured in the following way. In Section~\ref{sec:2} we present the probabilistic setup. In Section~\ref{sec:3} we define our new concept of forward transition rates. In Section~\ref{sec:review} we compare our definition with other suggestions in the literature, and we discuss how to partly `repair' the lack of match with transition densities.  In Section~\ref{sec:applic} we relate the work to actuarial practice. Section~\ref{sec:remarks} concludes.

\section{Setup and background}\label{sec:2}

Let $(\Omega,\mathbb{F},\mathbb{P})$ be some background probability space. Let $J\in\mathbb{N}$ and let $S=\{0,1,\ldots,J\}$ be some finite state space. In what follows, we consider a doubly-stochastic Markov setting where the state of the insured is described by a jump process (chain) with values in $S$. Instead of working with an abstract filtered probability space satisfying the usual conditions, we recall an explicit construction. Details can be found in \citet{Jacobsen2006}. The approach can be considered somewhat restrictive, but it allows a simpler and more concise discussion of forward transition rates.

\paragraph*{Notation and conventions}

Let $(Z_t)_{t \geq 0}$ be a stochastic process on $(\Omega,\mathbb{F},\mathbb{P})$ with values in some measurable space. We denote with $\mathbb{F}^Z := (\mathbb{F}^Z_t)_{t \geq 0}$ the natural filtration generated by $Z$ and to which it itself is adapted, i.e.\
\begin{align*}
\mathbb{F}_t^Z := \sigma( Z_s : s \leq t), \hspace{10mm} t\in[0,\infty).
\end{align*}
Furthermore, we define
\begin{align*}
\mathbb{F}_\infty^Z &:= \sigma\!\left( \bigcup_{t\geq 0} \mathbb{F}_t^Z\right)\!, \\
\mathbb{F}_{t+}^Z &:= \sigma(Z_s : s > t), \hspace{10mm} t \in [0,\infty).  
\end{align*}
We interpret $\mathbb{F}_\infty^Z$ as all the information generated by $Z$ and $\mathbb{F}_{t+}^Z$ as the future information generated by $Z$ (after time $t$).

In what follows, unless explicitly stated, all identities hold in an almost everywhere manner w.r.t.\ the probability measure $\mathbb{P}$.

\subsection{Doubly-stochastic Markov setting} \label{subsec:doubly}

For each possible transition $j,k\in S$, $k\neq j$, consider a stochastic process $[0,\infty) \ni t \mapsto \mu_{jk}(t)$ on $(\Omega,\mathbb{F},\mathbb{P})$ with values in $[0,\infty)$ and continuous sample paths. Using the Ionescu-Tulcea Theorem and the approach of \citet{Jacobsen2006} Chapter 3 and Section 7.2, we can construct a jump process $X:=(X_t)_{t\geq0}$ on $(\Omega,\mathbb{F},\mathbb{P})$ with values in $S$, which has a deterministic initial state $x_0 \in S$ and, conditionally on $\mu$, is Markovian with transition intensities $\mu$. Here we take $(\Omega,\mathbb{F},\mathbb{P})$ to be the canonical probability space associated with the construction.

That $X$ is Markovian conditionally on $\mu$ means that for all $t\in[0,\infty)$,
\begin{align*}
\mathbb{F}^X_{t+} \bigCI \mathbb{F}_t^X \mid \sigma(X_t)  \vee  \mathbb{F}^\mu_\infty,
\end{align*}
where $\sigma(X_t)  \vee  \mathbb{F}^\mu_\infty$ denotes the smallest $\sigma$-algebra that contains both $\sigma(X_t)$ and $ \mathbb{F}^\mu_\infty$. By construction, for $j,k\in S$, $k\neq j$, there exists (conditional) transition probabilities $P_{jk}^\mu$ such that for $0 \leq t < T < \infty$,
\begin{align*}
\mathbb{P}(X_T = k \, | \, \mathbb{F}_t^X \vee \mathbb{F}^\mu_\infty)
=
P_{X_t k}^\mu(t,T).
\end{align*}
Thus what we mean by the statement `$X$ conditionally on $\mu$ has transition intensities $\mu$', is that
\begin{align*}
\lim_{h \searrow 0} \frac{1}{h} P_{jk}^\mu(t,t+h)
=
\mu_{jk}(t),
\end{align*}
for all $t\in[0,\infty)$, which is well-defined as each $\mu_{jk}$ has continuous sample paths. Furthermore, the conditional transition probabilities $P_{jk}^\mu$ satisfy the Chapman-Kolmogorov equations and the \textit{backward and forward integral and differential equations} (the so-called Feller-Kolmogorov equations).

In the following, we assume that $\mathbb{E}[\mu_{jk}(t)]<\infty$ for all $j,k\in S$, $k\neq j$, and all $t\in[0,\infty)$.

\subsection{Preliminaries}

In general, $X$ is not unconditionally Markovian. An exception is whenever $\mu$ is deterministic; then $X$ is trivially Markovian with transition intensities $\mu$, and we recover the \textit{classic Markov chain life insurance setting}, see e.g.\ \citet{Hoem1969,Norberg1991}.

From now on fix a time-point $t\in[0,\infty)$. For valuation of future liabilities and pricing in pension and life insurance, interest lies in the expected accumulated cash flow, in particular expressions of the form $\mathbb{E}[ Z \, | \, \mathbb{F}_t^{X,\mu}]$, where $Z$ is some $\mathbb{F}_{t+}^{X,\mu}$/Borel-measurable random variable with values in $\mathbb{R}$ and finite expectation, $\mathbb{E}[|Z|]<\infty$. We think of $Z$ as a future payment. In this paper, we disregard the time value of money and market risks and focus exclusively on the expected accumulated cash flow. If the market risks are assumed to be independent of the biometric and behavioral risks, the following results and discussions immediately extend to valuation taking the time value of money into account. Details are given in Section~\ref{sec:remarks}, where dependency between market risks and biometric and behavioral risks is also briefly discussed.

Because $X$ is conditionally Markovian, we have the following results:
\begin{lemma}\label{lemma:cond}
It holds that
\begin{align*}
\mathbb{F}^{X,\mu}_{t+} &\bigCI \mathbb{F}_t^X \mid \sigma(X_t)  \vee  \mathbb{F}^\mu_t, \\
\mathbb{F}^\mu_{t+} &\bigCI \mathbb{F}_t^X \mid \mathbb{F}^\mu_t.
\end{align*}
\end{lemma}
\begin{proof}
See Appendix~\ref{ap:A}. 
\end{proof}
As an immediate consequence of the lemma, when $Z$ is $\mathbb{F}_{t+}^{X,\mu}$/Borel-measurable with values in $\mathbb{R}$ and $\mathbb{E}[|Z|]<\infty$, then
\begin{align} \label{eq:Z1}
\mathbb{E}[ Z \, | \, \mathbb{F}_t^{X,\mu}]
=
\mathbb{E}[ Z \, | \, \sigma(X_t)  \vee  \mathbb{F}^\mu_t].
\end{align}
If $Z$ furthermore is $\mathbb{F}_{t+}^\mu$-measurable, then
\begin{align} \label{eq:Z2}
\mathbb{E}[ Z \, | \, \mathbb{F}_t^{X,\mu}]
=
\mathbb{E}[ Z \, | \, \mathbb{F}^\mu_t].
\end{align}
Therefore, we are really interested in quantities in the form $\mathbb{E}[ Z \, | \, \sigma(X_t)  \vee  \mathbb{F}^\mu_t]$ or $\mathbb{E}[ Z \, | \, \mathbb{F}^\mu_t]$. These are by definition $\sigma(X_t)  \vee  \mathbb{F}^\mu_t$-measurable or $\mathbb{F}^\mu_t$-measurable, and we can therefore think of them as functions of the hitherto observed transition rates, and, if $Z$ is only $\mathbb{F}_{t+}^{X,\mu}$-measurable, also of the current state of the insured. 

\subsection{Forward mortality}

The concept of forward transition rates, which we introduce in the following section, is derived from the concept of forward mortality, which again is inspired by the concept of forward interest rates -- for details we refer to \citet{Norberg2010} Section 2-4. To motivate the discussion on forward transition rates, we now recall the concept of forward mortality.

Consider a jump process $X$ with values in $\{0,1\}$, which conditionally on $\mu$ is Markovian with transition intensities $\mu_{01}$ and $\mu_{10}=0$. This setting corresponds to a survival model with stochastic mortality $\mu_{01}$, see also Figure~\ref{fig:1}.
\begin{figure}[h]
	\centering
	\scalebox{0.8}{
	\begin{tikzpicture}[node distance=2em and 0em]
		\node[punkt] (0) {alive};
		\node[anchor=north east, at=(0.north east)]{$0$};
		\node[punkt, right = 30mm of 0] (1) {dead};
		\node[anchor=north east, at=(1.north east)]{$1$};
	\path
		(0)	edge [pil]		node [above]		{$\mu_{01}(\cdot)$}		(1)
	;
	\end{tikzpicture}}
	\caption{Survival model with stochastic mortality $\mu_{01}$.}
	\label{fig:1}
\end{figure}
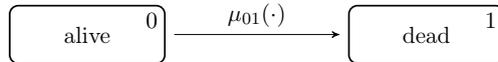
Various authors, including \citet{MilevskyPromislow2001}, \citet{Dahl2004}, and \citet{DahlMoller2006}, now essentially define the forward mortality rate as the $\mathbb{F}_t^\mu$-measurable and non-negative solution $(t,\infty) \ni T \mapsto m_{01}(t,T)$ to
\begin{align}\label{eq:forward_m}
\mathbb{E}\!\left[\left. e^{-\int_{(t,T]} \mu_{01}(s) \, \mathrm{d}s} \, \right| \mathbb{F}_t^\mu\right]
=
e^{-\int_{(t,T]} m_{01}(t,s) \, \mathrm{d}s}.
\end{align}
In what follows, we call $m_{01}$ defined by \eqref{eq:forward_m} the \textit{marginal forward mortality (rate)}, a choice of lingo which will become clear as we turn to the discussion of forward transition rate concepts in general.

Note that if we are allowed to interchange differentiation w.r.t.\ $T$ and integration w.r.t.\ $\mathbb{P}$ in \eqref{eq:forward_m}, the expression is equivalent to
\begin{align}\label{eq:forward_m_jump}
\mathbb{E}\!\left[\left. e^{-\int_{(t,T]} \mu_{01}(s) \, \mathrm{d}s} \mu_{01}(T) \, \right| \mathbb{F}_t^\mu\right]
=
e^{-\int_{(t,T]} m_{01}(t,s) \, \mathrm{d}s} m_{01}(t,T).
\end{align}
Consider now simple accumulated payments $[0,\infty) \ni s \mapsto B(s)$ given by
\begin{align*}
\mathrm{d}B(s)
&=
\mathds{1}_{(X_s = 0)} b_0(s) \, \mathrm{d}s
+
b_{01}(s) \, \mathrm{d}X_s, \hspace{5mm} s\in(0,\infty),\\
B(0)
&=
0,
\end{align*}
where $b_0$ and $b_{01}$ are continuous and real-valued deterministic functions. For valuation of future liabilities and pricing in pension and life insurance, interest lies in the expected accumulated cash flow, see e.g.\ \citet{BuchardtMoller2015}. A general definition suitable for our setup follows below; here $\mathbb{G}$ refers to some filtration on $(\Omega,\mathbb{F},\mathbb{P})$ containing all relevant information accessible to the valuator, and $B$ are accumulated payments assumed to be of finite variation, càdlàg and suitably integrable.
\begin{definition}\label{def:A}
Given information $\mathbb{G}$, the \textbf{expected accumulated cash flow} valuated at time $t\in[0,\infty)$ associated with the accumulated payments $B$ is defined by
\begin{align*}
(t,\infty) \ni T \mapsto A(t,T) 
:=
\mathbb{E}\!\left[B(T) - B(t) \left| \, \mathbb{G}_t\right.\right]\!.
\end{align*}
\end{definition}
We see that the expected accumulated cash flow valuated at time $t$ with relevant filtration $\mathbb{G}=\mathbb{F}^{X} \vee \mathbb{F}^\mu$ is given by
\begin{align}
A(t,T)
&=
\mathbb{E}\!\left[\left.
\mathbb{E}\!\left[B(T) - B(t) \left| \, \mathbb{F}_t^{X} \vee \mathbb{F}_\infty^\mu\right.\right]
\right| \mathbb{F}_t^{X} \vee \mathbb{F}_t^\mu\right] \nonumber \\
&=
\mathds{1}_{(X_t=0)}
\int_{(t,T]}
\mathbb{E}\!\left[\left.e^{-\int_{(t,s]} \mu_{01}(\tau) \, \mathrm{d}\tau}\left(b_0(s) + \mu_{01}(s)b_{01}(s)\right) \, \right| \mathbb{F}_t^\mu\right]
\mathrm{d}s \nonumber \\ \label{eq:Astoch}
&=
\mathds{1}_{(X_t=0)}
\int_{(t,T]}
e^{-\int_{(t,s]} m_{01}(t,\tau) \, \mathrm{d}\tau} \left(b_0(s) + m_{01}(t,s) b_{01}(s)\right) \mathrm{d}s,
\end{align}
where we have used the tower property, that $X$ conditionally on $\mu$ is Markovian with transition intensities $\mu_{01}$ and $\mu_{10}=0$, equation \eqref{eq:Z2}, and that $m_{01}$ satisfies \eqref{eq:forward_m} and \eqref{eq:forward_m_jump}. 

When $\mu$ is deterministic, we recover the classic setting and the expected accumulated cash flow reads
\begin{align} \label{eq:Adet}
\mathds{1}_{(X_t=0)}
\int_{(t,T]}
e^{-\int_{(t,s]} \mu_{01}(\tau) \, \mathrm{d}\tau} \left(b_0(s) + \mu_{01}(s) b_{01}(s)\right) \mathrm{d}s.
\end{align} 
Comparing \eqref{eq:Astoch} to \eqref{eq:Adet} reveals exactly the prowess of the marginal forward mortality: It allows one to calculate the expected accumulated cash flow in the usual manner regardless of the fact that the mortality is stochastic by replacing the stochastic mortality with the marginal forward mortality in standard formulae. The wish for similar results for multi-state models motivates the concept of forward transition rates, which we study in the following section.

\section{Forward equations rates}\label{sec:3}

In this section, we first provide a detailed exposition on the concept of forward transition rates for the doubly-stochastic Markov setting. In particular, we present the key properties which forward transition rate candidates desirably should satisfy. Motivated by this exposition, we next introduce the novel concept of \textit{forward equations rates} and hereby provide new insights regarding the possibility of generalizing the concept of forward mortality for multi-state models. In the next section, we compare the forward equations rates to previous forward transition rate definitions in the literature. This is done in both an abstract manner and through a detailed example for disability insurance.

\subsection{Forward transition rates}\label{subsec:ftr}

A natural question, as highlighted by \citet{Norberg2010}, is whether the concept of forward mortality can be adapted to and made fruitful in multi-state models, in particular the doubly-stochastic Markov setting, or whether the results surveyed in the previous paragraph rely on the specific structure of the survival model, in which case a generalization is unobtainable. In addition to the work of \citet{Norberg2010}, the question has also been investigated by e.g.\ \citet{ChristiansenNiemeyer} and \citet{Buchardt2017}.

To be more specific, the main question is if one can obtain similar replacement results regarding valuation as in the survival model. The main quantities of interest are
\begin{align*}
&\mathds{1}_{(X_T = k)}, \\ 
&\mathds{1}_{(X_{T-} = k)} \mu_{kl}(T), 
\end{align*}
where $T \in (t,\infty)$ and $k,l\in S, l \neq k$. Why these quantities? Consider e.g.\ simple accumulated payments $[0,\infty) \ni s \mapsto B(s)$ given by
\begin{align*}
\mathrm{d}B(s)
&=
b_{X_s}(s) \, \mathrm{d}s
+
\sum_{l \in S} \mathds{1}_{(X_{s-} \neq l)} b_{X_{s-}l}(s) \, \mathrm{d}N_{X_{s-}l}(s), \hspace{5mm} s\in(0,\infty),\\
B(0)
&=
0,
\end{align*}
where $N_{kl}$, $k,l \in S$, $l \neq k$, is the counting process counting the number of transitions from $k$ to $l$ for $X$, and $b_{k}$ and $b_{kl}$ are continuous real-valued deterministic functions describing the sojourn payments and payments upon transition, respectively. Omitting the technical details, it follows that $[0,\infty) \ni s \mapsto M(s)$ given by $M(0)=0$ and
\begin{align*}
M(s)
:=
B(s) - \int_{(0,s]}
\bigg(
\sum_{k \in S} \mathds{1}_{(X_u = k)} b_{k}(u)
+
\sum_{k,l \in S, l \neq k} \mathds{1}_{(X_{u-} = k)} \mu_{kl}(u) b_{kl}(u)
\bigg) \mathrm{d}u
\end{align*}
is a martingale w.r.t.\ the filtration $\mathbb{F}^X \vee \mathbb{F}_\infty^\mu$. By definition of the expected aggregated cash flow, it immediately becomes apparent why we are (solely, particularly) interested in the quantities $\mathds{1}_{(X_T = k)}$ and $\mathds{1}_{(X_{T-} = k)} \mu_{kl}(T)$.

Let $(t,\infty) \ni T \mapsto m_{jk}(t,T)$, $j,k \in S$, $k \neq j$, be some $\sigma(X_t) \vee \mathbb{F}_t^\mu$-measurable candidate forward transition rates. To fully generalize the replacement argument obtained in the survival model, one needs that there exists differentiable $\sigma(X_t)  \vee  \mathbb{F}^\mu_t$-measurable functions $[t,\infty) \ni T \mapsto P_{X_tk}^m(t,T)$, satisfying
\begin{align}\label{eq:kol_like}
\frac{\partial}{\partial T} P_{X_tk}^m(t,T)
&=
\sum_{l \neq k}
P_{X_tl}^m(t,T) m_{lk}(t,T)
-
P_{X_tk}^m(t,T) \sum_{l \neq k} m_{kl}(t,T), \hspace{5mm} k \neq X_t \\ \nonumber
\sum_{k \in S} P_{X_tk}^m(t,T) &= 1, \\ \nonumber
P_{X_tk}^m(t,t) &= \mathds{1}_{(X_t = k)}, \hspace{5mm} k \in S,
\end{align}
comparable to the Kolmogorov forward equations, such that
\begin{align}\label{eq:replace_m0}
&P_{X_tk}^m(t,T)
=
\mathbb{E}\!\left[ \mathds{1}_{(X_T = k)} \, | \, \sigma(X_t)  \vee  \mathbb{F}^\mu_t\right]\!, \\ \label{eq:replace_m00}
&P_{X_tk}^m(t,T) m_{kl}(t,T)
=
\mathbb{E}\!\left[ \mathds{1}_{(X_T = k)} \mu_{kl}(T) \, | \, \sigma(X_t)  \vee  \mathbb{F}^\mu_t\right]\!,
\end{align}
hold for all $k,l\in S$, $l\neq k$. Clearly, this boils down to two statements, that \eqref{eq:kol_like} and \eqref{eq:replace_m0} hold or that \eqref{eq:kol_like} and \eqref{eq:replace_m00} hold, or one stronger combined statement, namely that \eqref{eq:kol_like}, \eqref{eq:replace_m0}, and \eqref{eq:replace_m00} hold simultaneously. When referring to the first statement, we will often simply refer to \eqref{eq:replace_m0} on its own. In similar fashion, we also do not explicitly mention \eqref{eq:kol_like} when referring to the second or the combined statement.

The identities \eqref{eq:replace_m0} and \eqref{eq:replace_m00} are the cornerstones for our approach due the following reason. When \eqref{eq:replace_m0} holds, we have obtained successful replacement regarding the transition probabilities, while when \eqref{eq:replace_m00} holds, we have obtained successful replacement regarding the transition densities. Thus when they hold simultaneously, the expected accumulated cash flow is given by
\begin{align}\label{eq:Asotch2}
A(t,T)
&=
\mathbb{E}\!\left[B(T) - B(t) \left| \, \mathbb{F}^X_t \vee \mathbb{F}^\mu_t \right.\right] \nonumber \\
&=
\int_{(t,T]} \sum_{k \in S} P_{X_tk}^m(t,s) \bigg(
b_k(s) + \sum_{l \in S, l \neq k} m_{kl}(t,s) b_{kl}(s)
\bigg) \, \mathrm{d}s,
\end{align}
where we have employed similar techniques as in \eqref{eq:Astoch}, that $M$ is a martingale, and continuity of the (conditional) transition probabilities. The expected accumulated cash flow is then essentially in the `usual form' known from the classic Markov chain life insurance setting, the only difference being that the stochastic transition intensities have been replaced by the forward transition rates -- and thus a successful generalization of the replacement argument obtained in the survival model has been obtained.

Whenever $m$ is non-negative and continuous, one can think of $P_{X_t k}^m(t,\cdot)$ as transition probabilities for a Markovian jump process with initial state $X_t$ and transition intensities $m$ conditionally on all the information up until and including time $t$, compare to the construction of $X$ in the beginning of this Subsection~\ref{subsec:doubly}. As the forward transition rates $m$ in general are only $\sigma(X_t)  \vee  \mathbb{F}^\mu_t$-measurable, the transition intensities and thus also the transition probabilities for this Markovian jump process depend on the current state $X_t$. 

Note that we have yet to discuss existence and/or uniqueness of forward transition rate candidates satisfying \eqref{eq:kol_like}, \eqref{eq:replace_m0}, and \eqref{eq:replace_m00}. The identities just represent desirable properties for any definition of forward transition rates. In the next subsection, we introduce a novel forward transition rate candidate based directly on \eqref{eq:kol_like} and \eqref{eq:replace_m0}.

\subsection{Forward equations rates}\label{sec:fer}
In the following, we introduce a novel concept of \textit{forward equations rates} and discuss existence, uniqueness and other properties regarding this forward transition rate definition.

Define for each $j,k\in S$ the auxiliary $\mathbb{F}_t^\mu$-measurable function $[t,\infty) \ni T \mapsto \mathcal{P}_{jk}(t,T)$ by
\begin{align*}
 \mathcal{P}_{jk}(t,T) = \mathbb{E}\!\left[\left. P_{jk}^\mu(t,T) \, \right| \mathbb{F}^\mu_t\right]\!.
\end{align*}
We assume in the following that $\mathcal{P}_{jk}(t,\cdot)$ is differentiable for all $j,k \in S$. We then have the following forward transition rate definition.
\begin{definition}
Let $(t,\infty) \ni T \mapsto m_{jk}(t,T)$, $j,k\in S$, $k \neq j$, be $\mathbb{F}_t^\mu$-measurable. If the following system of equations are satisfied,
\begin{align}\label{eq:forward_eqs_def}
\frac{\partial}{\partial T} \mathcal{P}_{jk}(t,T)
=
\sum_{l \neq k}
\mathcal{P}_{jl}(t,T) m_{lk}(t,T)
-
\mathcal{P}_{jk}(t,T) \sum_{l \neq k} m_{kl}(t,T), \hspace{5mm} j,k \in S, k \neq j,
\end{align}
we say that $m$ are \textbf{forward equations rates} for $X$.
\end{definition}
This definition is similar to one suggested by \citet{Norberg2010}, but there is a single but crucial difference. Norberg essentially suggests to define the forward transition rates as the $\sigma(X_t) \vee \mathbb{F}_t^\mu$-measurable solution to the system of equations
\begin{align*}
\frac{\partial}{\partial T} \mathbb{E}\!\left[ \mathds{1}_{(X_T = k)} \, | \, \sigma(X_t)  \vee  \mathbb{F}^\mu_t\right]
=
&\sum_{l \neq k}
\mathbb{E}\!\left[ \mathds{1}_{(X_T = l)} \, | \, \sigma(X_t)  \vee  \mathbb{F}^\mu_t\right] m_{lk}(t,T) \\
&-
\mathbb{E}\!\left[ \mathds{1}_{(X_T = k)} \, | \, \sigma(X_t)  \vee  \mathbb{F}^\mu_t\right] \sum_{l \neq k} m_{kl}(t,T), \, k \neq X_t,
\end{align*}
which is just \eqref{eq:kol_like} combined with \eqref{eq:replace_m0}. The definition imposed by \eqref{eq:forward_eqs_def} can be seen as an extension involving all transition probabilities rather than only those related to the present state of the insured, $X_t$. As such, \eqref{eq:forward_eqs_def} is a constrained version of Norberg's definition requiring the equations to hold for a portfolio of insured with different present states covering all states. As noted by Norberg, his system of equations consists of $J(J-1)$ unknowns but only $(J-1)$ equations, which in general would lead to infinitely many solutions. In comparison, \eqref{eq:forward_eqs_def} consists of $J(J-1)$ equations, so we actually expect the forward equations rates to exist and be unique (under suitable regularity conditions). 

Together with the (trivially satisfied) conditions
\begin{align}\label{eq:conds}
\sum_{k \in S} \mathcal{P}_{jk}(t,T) = 1, \hspace{5mm} \mathcal{P}_{jk}(t,t) = \mathds{1}_{(j = k)},
\end{align}
it can be shown that the forward equations rates are defined exactly such that \eqref{eq:kol_like} and \eqref{eq:replace_m0} hold when setting $P_{X_tk}^m(t,\cdot)=\mathcal{P}_{j k}(t,\cdot)$ on $(X_t = j)$ for any $j\in S$ and all $k\in S$.

By definition, the forward equations rates are $\mathbb{F}_t^\mu$-measurable: They are not allowed to depend on the current state $X_t$. Later, when we investigate forward transition rate concepts in the literature, we return to a discussion of pros and cons regarding this property.

In general, the forward equations rates are allowed to be negative: In this case, one cannot think of $\mathcal{P}_{jk}(t,\cdot)$ as transition probabilities for some Markovian jump process, but must think of them as the solution to a system of differential equations similar to the Kolmogorov forward equations, namely \eqref{eq:forward_eqs_def} with conditions \eqref{eq:conds}.

Regarding existence and uniqueness of the forward equations rates we have the following result, which is applicable for so-called decrement models, where return to a state is not possible once it has been left. Examples include the disability model without recovery with and without policyholder behavior (free-policy and surrender).

\begin{theorem}\label{thm:existsunique}
Assume that $\mathcal{P}_{jk}(t,\cdot)$ is differentiable for all $j,k \in S$ and that $\mathcal{P}_{jk}(t,\cdot) = 0$ for $k < j$. Then the forward equations rates exist and are unique. If furthermore $\mathcal{P}_{jk}(t,\cdot)$ is continuously differentiable for all $j,k \in S$, then the forward equations rates are continuous.
\end{theorem}
\begin{proof}
See Appendix~\ref{ap:A}.
\end{proof}

Regarding further properties of the forward equations rates, we note the following.

To conclude that $m_{jk}(t,\cdot)=0$, one needs not only that direct transition from $j$ to $k$ is impossible but also that indirect transition from $j$ to $k$ is impossible. In other words, $\mu_{jk}=0$ does not imply $m_{jk}(t,\cdot)=0$ unless the stronger requirement $\mathcal{P}_{jk}(t,\cdot)=0$ holds. This can be verified by e.g.\ considering a disability model without recovery and active-mortality equal to zero. In particular, \eqref{eq:replace_m00} does not hold in general. Rather, the closest obtainable identity involves the difference between the sum over all transitions to and from, respectively, each state. To be rigorous, it follows under the assumption of interchangeable differentiation w.r.t.\ $T$ and integration w.r.t.\ $\mathbb{P}$, that
\begin{align}\label{eq:sum}
&\sum_{l \neq k} P_{X_t l}^m(t,T)m_{lk}(t,T)
-
\sum_{l \neq k} P_{X_t k}^m(t,T) m_{kl}(t,T) \\
=
&\sum_{l \neq k} \mathbb{E}\!\left[ \mathds{1}_{(X_T = l)} \mu_{lk}(T) \, | \, \sigma(X_t)  \vee  \mathbb{F}^\mu_t\right]
-
\sum_{l \neq k}  \mathbb{E}\!\left[ \mathds{1}_{(X_T = k)} \mu_{kl}(T) \, | \, \sigma(X_t)  \vee  \mathbb{F}^\mu_t\right]\!, \nonumber
\end{align}
for $k \in S$ using the definition of the forward equations rates $m$ and the (conditional) Kolmogorov forward equations for $X$. This exactly shows that \eqref{eq:replace_m00} holds only for the difference between the sum over transitions to and from, respectively, state $k$. For some specific classes of models this implies \eqref{eq:replace_m00}. To see this for competing risks models, note that for each death-state there is only one relevant transition, namely transition to this state from the alive-state, as the remaining transition intensities associated with the state are zero, whereby the above identity is identical to \eqref{eq:replace_m00}.

Because only \eqref{eq:replace_m0} holds in general for the forward equations rates, the replacement argument of the survival model cannot be generalized fully. However, if the insurance contract does not contain payments upon transition, i.e.\ if $b_{kl} = 0$ for $k,l\in S$, $l\neq k$, then only \eqref{eq:replace_m0} is required and the replacement argument generalizes. Thus if one is only interested in valuation of sojourn payments, the forward equations rates are a fruitful starting point.

Because the forward equations rates are defined directly from \eqref{eq:replace_m0}, the above discussion shows that any general definition of forward transition rates that is to satisfy both \eqref{eq:replace_m0} and \eqref{eq:replace_m00} cannot be $\mathbb{F}_t^\mu$-measurable (and thus must be allowed to depend on the current state $X_t$). This motivates the forward transition rate definition of \citet{Buchardt2017} which we discuss in the following section.

\section{Forward transition rate definitions in the literature}\label{sec:review}

We now review the contributions of \citet{ChristiansenNiemeyer} and \citet{Buchardt2017} in comparison with the \textit{forward equations rates}, which reveals strengths and weaknesses of each individual forward transition rate definition and leads to new insights regarding the forward transition rate concept in itself. 

\subsection{Alternative definitions from the literature}

\paragraph*{Christiansen \& Niemeyer}

In \citet{ChristiansenNiemeyer}, forward transition rates are not discussed independently of the financial market, but this can easily be done by taking interest rate zero in their setting. Christiansen \& Niemeyer define forward rates implicitly, essentially requiring they allow for replacement arguments comparable to ours for a set of insurance products. Based on the specific multi-state models and insurance products they consider, these requirements suggest the following definition: for each $j,k\in S$, $k\neq j$, the forward rate for this transition is the $\mathbb{F}_t^\mu$-measurable and non-negative solution $(t,\infty)\ni T \mapsto m_{jk}(t,T)$ to
\begin{align}\label{eq:forward_chrN}
\mathbb{E}\!\left[\left. e^{-\int_{(t,T]} \mu_{jk}(s) \, \mathrm{d}s} \right| \, \mathbb{F}_t^\mu\right]
=
e^{-\int_{(t,T]} m_{jk}(t,s) \, \mathrm{d}s}.
\end{align}
We note that $m_{jk}$ does not depend on the current state of the insured: it only depends on the hitherto observed transition rates. Furthermore, the definition of $m_{jk}$ does not involve the structure of the jump process $X$: they are `universal'. In particular, the marginal forward mortality given by \eqref{eq:forward_m} is a special case of the general definition of \eqref{eq:forward_chrN}. To see this, let $X$ be a conditionally Markovian jump process given $\mu$ with values in $\{0,1\}$ with $\mu_{10}=0$. Then \eqref{eq:forward_chrN} is just \eqref{eq:forward_m} in disguise. Therefore, we call $m_{jk}$ defined by \eqref{eq:forward_chrN} the \textit{marginal forward transition rate}, as it solely relies on the probabilistic structure of $\mu$. Consequentially, the marginal forward transition rates are particularly restrictive and idealistic. 

Whenever the marginal forward transition rates satisfy \eqref{eq:replace_m0}, they agree with the forward equations rates.
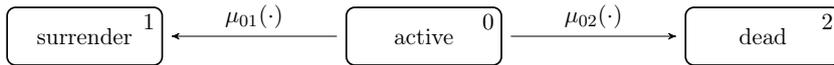
\begin{figure}[h]
	\centering
	\scalebox{0.8}{
	\begin{tikzpicture}[node distance=2em and 0em]
		\node[punkt] (0) {active};
		\node[anchor=north east, at=(0.north east)]{$0$};
		\node[punkt, left = 30mm of 0] (1) {surrender};
		\node[anchor=north east, at=(1.north east)]{$1$};
		\node[punkt, right = 30mm of 0] (2) {dead};
		\node[anchor=north east, at=(2.north east)]{$2$};
	\path
		(0)	edge [pil]		node [above]		{$\mu_{01}(\cdot)$}		(1)
		(0)	edge [pil]		node [above]		{$\mu_{02}(\cdot)$}		(2)
	;
	\end{tikzpicture}}
	\caption{Active-surrender-dead model with transition rates $\mu_{01}$ and $\mu_{02}$.}
	\label{fig:2}
\end{figure}
To see this, consider an active-surrender-dead model with transition rates from active to surrender and active to dead as in Figure~\ref{fig:2}. 
Then on $(X_t = 0)$ and with $k=0$, we can restate \eqref{eq:replace_m0} as
\begin{align*}
e^{-\int_{(t,T]} \left(m_{01}(t,s) + m_{02}(t,s)\right) \, \mathrm{d}s}
=
\mathbb{E}\!\left[\left. e^{-\int_{(t,T]} \left(\mu_{01}(s) + \mu_{02}(s)\right) \, \mathrm{d}s} \, \right| \mathbb{F}_t^\mu\right]
\end{align*}
On the other hand, by definition of the marginal forward transition rates, i.e.\ \eqref{eq:forward_chrN},
\begin{align*}
e^{-\int_{(t,T]} m_{01}(t,s) \, \mathrm{d}s}
&e^{-\int_{(t,T]} m_{02}(t,s) \, \mathrm{d}s} \\
&=
\mathbb{E}\!\left[\left. e^{-\int_{(t,T]} \mu_{01}(s) \, \mathrm{d}s} \, \right| \mathbb{F}_t^\mu\right]
\mathbb{E}\!\left[\left. e^{-\int_{(t,T]} \mu_{02}(s) \, \mathrm{d}s} \, \right| \mathbb{F}_t^\mu\right]\!.
\end{align*}
Collecting, we obtain the identity
\begin{align*}
\mathbb{E}\!\left[\left. e^{-\int_{(t,T]} \left(\mu_{01}(s) + \mu_{02}(s)\right) \, \mathrm{d}s} \, \right| \mathbb{F}_t^\mu\right]
=
\mathbb{E}\!\left[\left. e^{-\int_{(t,T]} \mu_{01}(s) \, \mathrm{d}s} \, \right| \mathbb{F}_t^\mu\right]
\mathbb{E}\!\left[\left. e^{-\int_{(t,T]} \mu_{02}(s) \, \mathrm{d}s} \, \right| \mathbb{F}_t^\mu\right]
\end{align*}
which, unless $\mu_{01}$ and $\mu_{02}$ are independent, is not satisfied in general, see also \citet{ChristiansenNiemeyer} Subsection 6.2 with interest rate zero. The situation is fully comparable to the discussion of forward mortalities and interest rates in the case of dependency between the biometric risks and the financial market, see e.g.\ \citet{ChristiansenNiemeyer} Subsection 6.1, \citet{MiltersenPersson}, and \citet{Buchardt2014}.

\citet{ChristiansenNiemeyer} consider a large class of diffusion processes for the transitions rates $\mu$ and show the equivalence between specific dependency structures and the identities \eqref{eq:replace_m0} and \eqref{eq:replace_m00} for a number of multi-state models, including a disability model without recovery. Hereby, they show that if one desires `universal' forward transitions rates that solely rely on the probabilistic structure of $\mu$, such as the marginal forward transition rates, one must assume a specific and often unrealistic dependency structure between the transitions rates, see e.g.\ \citet{ChristiansenNiemeyer} paragraph following Remark 5.5. If instead one is solely interested in sojourn payments and willing to specify a specific structure of the jump process $X$, the forward equations rates provide a natural alternative not confined to a specific dependency structure between the transition intensities.

\paragraph*{Buchardt}

The definition of forward transition rates studied by \citet{Buchardt2017} is for all practical purposes, see also \citet{Buchardt2017} Lemma 4.3, equivalent to setting
\begin{align}\label{eq:buc_forward}
(t,\infty) \ni T \mapsto m_{kl}(t,T)
:=
\frac{
\mathbb{E}\!\left[ \mathds{1}_{(X_T = k)} \mu_{kl}(T) \, | \, \sigma(X_t)  \vee  \mathbb{F}^\mu_t\right]
}
{
\mathbb{E}\!\left[ \left.\mathds{1}_{(X_T = k)} \, \right| \, \sigma(X_t)  \vee  \mathbb{F}^\mu_t\right]
}
\geq 0
\end{align}
for all $k,l \in S$, $l\neq k$, whenever the right-hand side is well-defined. This definition was already proposed by \citet{Norberg2010} Section 6, final paragraph.

We observe that if $\mu_{kl}=0$, then $m_{kl}(t,\cdot)=0$. Furthermore, from \citet{Buchardt2017} Theorem 4.4 it follows that \eqref{eq:replace_m0} holds (under some minor regularity conditions), such that by definition and rearrangement also \eqref{eq:replace_m00} is satisfied. On the other hand, contrary to the forward equations rates and the marginal forward transition rates, the forward transition rates of \eqref{eq:buc_forward} can by definition generally not be taken to be $\mathbb{F}_t^\mu$-measurable but must be allowed to depend on the current state $X_t$. Therefore, we call $m_{kl}$ defined by \eqref{eq:buc_forward} the \textit{state-wise forward transition rate}.

In competing risks models, the state-wise forward transition rates agree with the forward equations rates (but in general differ from the marginal forward transitions rates unless the transition rates are assumed to be independent). If one imposes a specific structure on the transition intensities, this result can be extended beyond competing risks models -- see also the example at the end of  Section~\ref{sec:applic}).

In the following subsection, similarities and differences between the state-wise forward transition rates and the forward equations rates are exemplified in the context of disability insurance. Furthermore, we exemplify how suitable state space and payment process `tweaks' might allow for valuation of transition payments using the forward equations rates (by `repairing' the lack of match with transition densities).

\subsection{Disability insurance -- `repairing' the forward equations rates}\label{subsec:repair}
Consider a disability model without recovery as in Figure~\ref{fig:3}. 
\begin{figure}[h]
	\centering
	\scalebox{0.8}{
	\begin{tikzpicture}[node distance=2em and 0em]
		\node[punkt] (1) {disabled};
		\node[anchor=north east, at=(1.north east)]{$1$};
		\node[punkt, left = 30mm of 1] (0) {active};
		\node[anchor=north east, at=(0.north east)]{$0$};
		\node[punkt, below = 20mm of 1] (2) {dead};
		\node[anchor=north east, at=(2.north east)]{$2$};
	\path
		(0)	edge [pil]		node [above]			{$\mu_{01}(\cdot)$}		(1)
		(1)	edge [pil]		node [right]			{$\mu_{12}(\cdot)$}		(2)
		(0)	edge [pil]		node [below left]		{$\mu_{02}(\cdot)$}		(2)
	;
	\end{tikzpicture}}
	\caption{Disability model without recovery.}
	\label{fig:3}
\end{figure}
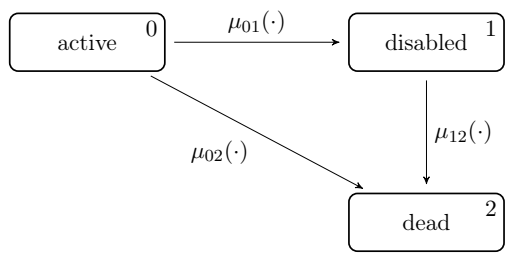
The insurance contract we have in mind is one stipulated by 
\begin{itemize}
\item Premium payments when active, financing:
\begin{itemize}
\item Disability coverage, including:
\begin{itemize}
\item Payment upon transition from active to disabled.
\item Sojourn payments when disabled.
\end{itemize}
\item Death coverage given by payments upon transition to the state dead.
\end{itemize}
\end{itemize}

\paragraph*{State-wise rates and forward equations rates}
Consider the state-wise forward transition rates from \eqref{eq:buc_forward}. These depend on the current state of the insured, thus we denote them by $m^{X_t}(t,\cdot)$. As discussed previously, both \eqref{eq:replace_m0}  and \eqref{eq:replace_m00} are satisfied by the state-wise forward transition rates.

In general, $m^{0}_{12}(t,\cdot)$ and $m^{1}_{12}(t,\cdot)$ differ. On $(X_t=0)$, i.e.\ when the insured is active at the present time, valuation of future sojourn payments and payments upon transition can be performed in a Markov model with $m^{0}(t,\cdot)$ as transition rates, see Figure~\ref{fig:disabled_buc} (left). On $(X_t=1)$, i.e.\ when the insured is presently disabled, valuation must be performed in a Markov model with different transition rates $m^{1}(t,\cdot)$, see Figure~\ref{fig:disabled_buc} (right).
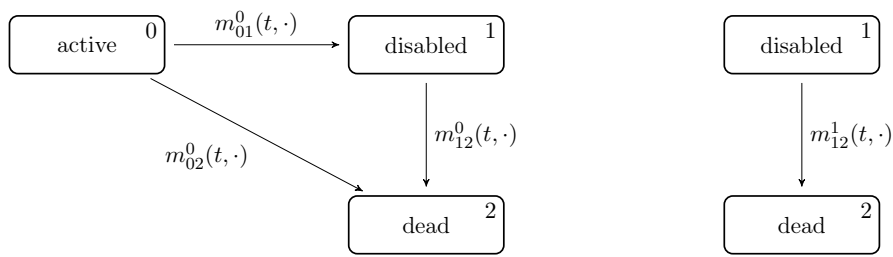
\begin{figure}[h]
	\centering
	\scalebox{0.8}{
	\begin{tikzpicture}[node distance=2em and 0em]
		\node[punkt] (1) {disabled};
		\node[anchor=north east, at=(1.north east)]{$1$};
		\node[punkt, left = 30mm of 1] (0) {active};
		\node[anchor=north east, at=(0.north east)]{$0$};
		\node[punkt, below = 20mm of 1] (2) {dead};
		\node[anchor=north east, at=(2.north east)]{$2$};
	\path
		(0)	edge [pil]		node [above]			{$m^{0}_{01}(t,\cdot)$}		(1)
		(1)	edge [pil]		node [right]			{$m^{0}_{12}(t,\cdot)$}		(2)
		(0)	edge [pil]		node [below left]		{$m^{0}_{02}(t,\cdot)$}		(2)
	;
	\end{tikzpicture}
	\hspace{30mm}
	\begin{tikzpicture}[node distance=2em and 0em]
		\node[punkt] (1) {disabled};
		\node[anchor=north east, at=(1.north east)]{$1$};
		\node[punkt, below = 20mm of 1] (2) {dead};
		\node[anchor=north east, at=(2.north east)]{$2$};
	\path
		(1)	edge [pil]		node [right]			{$m^{1}_{12}(t,\cdot)$}		(2)
	;
	\end{tikzpicture}}
	\caption{Two Markov models with state-wise forward transition rates replacing the doubly-stochastic Markov disability model without recovery: one to be used when the insured is presently active (left) and another to be used when the insured is presently disabled (right).}
	\label{fig:disabled_buc}
\end{figure}
In particular, four rather than three non-zero transition rates are required. From a practical and implementational point of view, valuation therefore remains slightly more complicated than in the classic Markov chain life insurance setting. Furthermore, the dependency of the forward transition rates on the current state of the insured makes them difficult to interpret.

Consider now instead the forward equations rates which we in a slight abuse of notation denote by $m(t,\cdot)$. As discussed previously, \eqref{eq:replace_m00} is in general not satisfied by the forward equations rates -- and this is also the case for the disability model without recovery. But valuation of future sojourn payments can be performed in a Markov model with $m(t,\cdot)$ as transition rates, see Figure~\ref{fig:disabled_forward}.
\begin{figure}[h]
	\centering
	\scalebox{0.8}{
	\begin{tikzpicture}[node distance=2em and 0em]
		\node[punkt] (1) {disabled};
		\node[anchor=north east, at=(1.north east)]{$1$};
		\node[punkt, left = 30mm of 1] (0) {active};
		\node[anchor=north east, at=(0.north east)]{$0$};
		\node[punkt, below = 20mm of 1] (2) {dead};
		\node[anchor=north east, at=(2.north east)]{$2$};
	\path
		(0)	edge [pil]		node [above]			{$m_{01}(t,\cdot)$}		(1)
		(1)	edge [pil]		node [right]			{$m_{12}(t,\cdot)$}		(2)
		(0)	edge [pil]		node [below left]		{$m_{02}(t,\cdot)$}		(2)
	;
	\end{tikzpicture}}
	\caption{Markov model with forward equations rates replacing the doubly-stochastic Markov disability model without recovery for valuation of sojourn payments.}
	\label{fig:disabled_forward}
\end{figure}
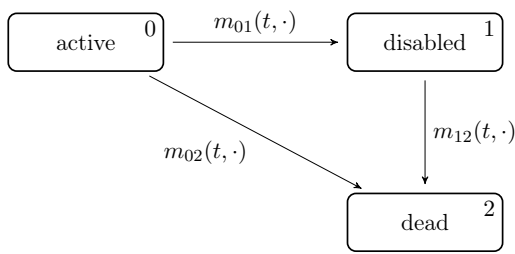
It can be shown that $m_{12}(t,\cdot) = m^{1}_{12}(t,\cdot)$ and that \eqref{eq:replace_m00} does hold for the forward equations rates on $(X_t = 1)$, confer with \eqref{eq:sum}. But in general, \eqref{eq:replace_m00} does not hold on $(X_t = 0)$, in particular, the forward equations rates will not allow one to valuate transition payments from active to disabled or active to dead. To summarize, only parts of the original disability insurance contract we had in mind can be valuated if one insists on using forward equations rates. On the other hand, these parts -- including the premiums payments when active and the sojourn payments when disabled -- can be handled inside the technical and/or numerical framework of the classic Markov chain life insurance setting.

\paragraph*{`Repairing' the forward equations rates}
We end this subsection by describing a way to tweak the original model slightly that extends the area of applicability of the forward equations rates.

Consider a new jump process $\tilde{X}$ defined from $X$ by adding separate death states as in Figure~\ref{fig:disabled2}. 
\begin{figure}[h]
	\centering
	\scalebox{0.8}{
	\begin{tikzpicture}[node distance=2em and 0em]
		\node[punkt] (1) {disabled};
		\node[anchor=north east, at=(1.north east)]{$1$};
		\node[punkt, left = 30mm of 1] (0) {active};
		\node[anchor=north east, at=(0.north east)]{$0$};
		\node[punkt, below = 20mm of 0] (3) {dead};
		\node[anchor=north east, at=(3.north east)]{$3$};
		\node[punkt, below = 20mm of 1] (2) {dead\textquotesingle};
		\node[anchor=north east, at=(2.north east)]{$2$};
	\path
		(0)	edge [pil]		node [above]			{$\mu_{01}(\cdot)$}		(1)
		(1)	edge [pil]		node [right]			{$\mu_{12}(\cdot)$}		(2)
		(0)	edge [pil]		node [left]			{$\mu_{02}(\cdot)$}		(3)
	;
	\end{tikzpicture}}
	\caption{Disability model without recovery but with separate death states.}
	\label{fig:disabled2}
\end{figure}
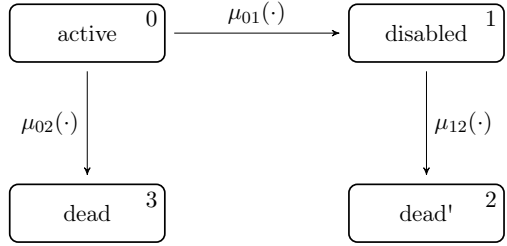
For all practical purposes, the models with and without separate death states are interchangeable as long as the sojourn payments in the two death states do not differ. To be rigorous, the new jump process satisfies 
\begin{align*}
\tilde{X}_t
=
\mathds{1}_{(X_t \in \{0,1\})} X_t
+
\mathds{1}_{(N_{12}(t) = 1)} 2
+
\mathds{1}_{(N_{02}(t) = 1)} 3,
\end{align*}
where $N$ is the multivariate counting process associated with $X$. Define $\tilde{\mu}$ as the corresponding (conditional) transition intensities, such that e.g.\ $\tilde{\mu}_{03} = \mu_{02}$, $\tilde{\mu}_{02} = 0$, and $\tilde{\mu}_{12} = \mu_{12}$.  Then $\tilde{X}$ is also conditionally Markovian given $\tilde{\mu}$, as described initially in Subsection~\ref{subsec:doubly}, but contains a separate death state for death after disability. 

One can show that while the state-wise forward transition rates remain unaffected, the forward equations rates for $X$ and $\tilde{X}$ differ. Denote the latter by $\tilde{m}(t,\cdot)$. We cannot in general conclude that the forward equations rate $\tilde{m}_{02}(t,\cdot)$ is zero because indirect transition from state $0$ to state $2$ remains possible, which leads to the Markov model of Figure~\ref{fig:disabled2_forward}.
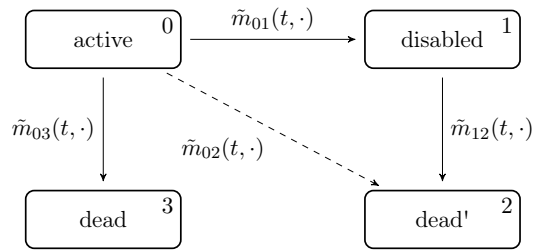
\begin{figure}[h]
	\centering
	\scalebox{0.8}{
	\begin{tikzpicture}[node distance=2em and 0em]
		\node[punkt] (1) {disabled};
		\node[anchor=north east, at=(1.north east)]{$1$};
		\node[punkt, left = 30mm of 1] (0) {active};
		\node[anchor=north east, at=(0.north east)]{$0$};
		\node[punkt, below = 20mm of 0] (3) {dead};
		\node[anchor=north east, at=(3.north east)]{$3$};
		\node[punkt, below = 20mm of 1] (2) {dead\textquotesingle};
		\node[anchor=north east, at=(2.north east)]{$2$};
	\path
		(0)	edge [pil]		node [above]			{$\tilde{m}_{01}(t,\cdot)$}		(1)
		(1)	edge [pil]		node [right]			{$\tilde{m}_{12}(t,\cdot)$}		(2)
		(0)	edge [pil]		node [left]			{$\tilde{m}_{03}(t,\cdot)$}		(3)
		(0)	edge [pildashed]	node [below left]		{$\tilde{m}_{02}(t,\cdot)$}		(2)
	;
	\end{tikzpicture}}
	\caption{Markov model with forward equations rates replacing the doubly-stochastic Markov disability model without recovery with separate death states for alternative valuation of sojourn payments. Note the non-zero transition rate $\tilde{m}_{02}(t,\cdot)$ even though $\tilde{\mu}_{02}(\cdot)=0$.}
	\label{fig:disabled2_forward}
\end{figure}
In general, \eqref{eq:replace_m00} does not hold on $(X_t=0)$. Though when also $k=0$ and $l=3$, corresponding to valuation of payments upon transition from active to dead, \eqref{eq:replace_m00} is satisfied, see e.g.\ \eqref{eq:sum}. For valuation of payments upon transition from disabled to dead when the insured is presently active, we can rewrite \eqref{eq:sum} and obtain the following on $(X_t = 0)$:
\begin{align*}
\mathbb{E}\!\left[ \mathds{1}_{(X_T = 1)} \mu_{12}(T) \, | \, \sigma(X_t)  \vee  \mathbb{F}^\mu_t\right]
&=
\mathbb{E}\!\left[ \mathds{1}_{(\tilde{X}_T = 1)} \tilde{\mu}_{12}(T) \, | \, \sigma(\tilde{X}_t)  \vee  \mathbb{F}^{\tilde{\mu}}_t\right] \\
&=
\tilde{P}_{\tilde{X}_t0}^{\tilde{m}}(t,T)\tilde{m}_{02}(t,T)
+
\tilde{P}_{\tilde{X}_t1}^{\tilde{m}}(t,T)\tilde{m}_{12}(t,T).
\end{align*}
Thus for accumulated payments given by
\begin{align*}
\mathrm{d}B^{(1)}(s)
&=
\mathds{1}_{(X_{s-} = 1)} b_{12}(s) \, \mathrm{d}N_{X_{s-}2}(s), \hspace{5mm} s\in(0,\infty),\\
B^{(1)}(0)
&=
0,
\end{align*}
corresponding exactly to payment $b_{12}$ upon transition from disabled to dead, the expected accumulated cash flow can on $(X_t = 0)$ be written as
\begin{align*}
A^{(1)}(t,T)
=
\int_{(t,T]}
\left(\tilde{P}_{00}^{\tilde{m}}(t,s)\tilde{m}_{02}(t,s)b_{12}(s)
+
\tilde{P}_{01}^{\tilde{m}}(t,s)\tilde{m}_{12}(t,s)b_{12}(s)\right) \mathrm{d}s.
\end{align*}
Thus valuation of the payments given by $B^{(1)}$ can be performed in the Markov model of Figure~\ref{fig:disabled2_forward} with $\tilde{m}(t,\cdot)$ as transition rates through valuation of a different payment process  with payment $b_{12}$ upon transition from disabled to dead\textquotesingle\,\,as well as payment $b_{12}$ upon transition from active to dead\textquotesingle.

Similar arguments apply for the payments upon transition from active to disabled. Here valuation can also be performed in a Markov model with $\tilde{m}(t,\cdot)$ as transition rates through valuation of a different payment process with payment $b_{01}$ upon transition from active to disabled as well as payment $b_{01}$ upon transition from active to disabled-dead. Thus all parts of the original disability insurance contract we had in mind can be valuated in a Markov model, namely that of Figure~\ref{fig:disabled2_forward}, using forward equations rates if (and only if) one is willing to tweak the setup suitably. In particular, four rather than three non-zero transition rates are required. This means that from a practical and implementation point of view, valuation of payments upon transition remains slightly more complicated than in the classic Markov chain life insurance setting. Furthermore, the resulting forward transition rates are difficult to interpret.

Whether one works with state-wise forward transition rates or forward equations rates, we can conclude that four rather than three non-zero transition rates are required for the disability model without recovery. On the other hand, the above arguments do not generalize to arbitrary (non-decrement) models but rely extensively on the (decrement) structure of the disability model without recovery. So while it seems equally demanding to implement forward equations rates and state-wise forward transition rates (recall Figure~\ref{fig:disabled_buc}) for valuation in the disability model without recovery, only implementation of the latter has a natural generalization to the most advanced models.

\subsection{Summary and model calibration}

All definitions discussed in the previous subsections extend the concept of forward mortality rates to a multi-state framework and contain the marginal forward mortality as a special case. The properties of the various forward transition rate definitions are summarized in Table~\ref{tbl:compare}.
\begin{table}[h]
\centering
\renewcommand{\arraystretch}{1.5}
\small{
\begin{tabularx}{0.85\linewidth}{X|cccc}
& Universal & $\mathbb{F}_t^\mu$-measurable & \eqref{eq:replace_m0} & \eqref{eq:replace_m00} \\ \hline
Marginal rates & \cmark & \cmark & & \\
Forward equations rates & & \cmark & \cmark & \\
State-wise rates & & & \cmark & \cmark \\
\end{tabularx}}
\caption{Comparison of properties of the different forward transition rate definitions. Here the definition is said to be `Universal' if it does not depend on the specific structure of the jump process but only relies on the probabilistic structure of the transition rates.}
\label{tbl:compare}
\end{table}

The extensions all express different ambitions. The definition of the marginal forward transitions rates desires a sort of `universality', in the sense that this definition does not rely on the specific structure of the state space or distribution of $X$ but only relies relies on the probabilistic structure of $\mu$. In general, this will not lead to a successful replacement argument, neither for sojourn payments, consult \eqref{eq:replace_m0}, nor payments upon transition, consult \eqref{eq:replace_m00}. In the definition of the forward equations rates, this condition is relaxed and only $\mathbb{F}_t^\mu$-measurability is required, such that the rates still do not depend on the current state $X_t$ of the insured. The replacement argument is then successful for sojourn payments but not in general for payments upon transition. Finally, the state-wise forward transition rates are allowed to depend on the current state of the insured, in which case the replacement argument holds for both sojourn payments and payments upon transition.

Another point of comparison between the definitions consists of comparing the quantities needed for a calibration similar to that of forward mortalities and forward interest rates.

To calibrate the marginal forward transition rates, we require the quantities
\begin{align*}
(t,\infty) \ni T \mapsto \mathbb{E}\!\left[\left. e^{-\int_{(t,T]} \mu_{jk}(s) \, \mathrm{d}s} \, \right| \mathbb{F}_t^\mu\right]
\end{align*}
for $j,k\in S$, $k\neq j$. These quantities are not directly linked to any insurance contracts in the market.

To calibrate the forward equations rates, we require the quantities
\begin{align*}
(t,\infty) \ni T \mapsto \mathcal{P}_{jk}(t,T)=\mathbb{E}\!\left[\left. P_{jk}^\mu(t,T) \, \right| \mathbb{F}^\mu_t\right]
\end{align*}
for $j,k \in S$. Assuming interest rate zero, these quantities are directly linked to insurance contracts consisting of sojourn payments.

To calibrate the state-wise forward transition rates, we require the quantities
\begin{align*}
&(t,\infty) \ni T \mapsto \mathbb{E}\!\left[ \mathds{1}_{(X_T = k)} \, | \, \sigma(X_t)  \vee  \mathbb{F}^\mu_t\right]
=
\mathbb{E}\!\left[P_{X_tk}^\mu(t,T) \, | \, \sigma(X_t)  \vee  \mathbb{F}^\mu_t\right]\!, \\
&(t,\infty) \ni T \mapsto \mathbb{E}\!\left[ \mathds{1}_{(X_T = k)} \mu_{kl}(T) \, | \, \sigma(X_t)  \vee  \mathbb{F}^\mu_t\right]
=
\mathbb{E}\!\left[ P_{X_tk}^\mu(t,T) \mu_{kl}(T) \, | \, \sigma(X_t)  \vee  \mathbb{F}^\mu_t\right]
\end{align*}
for $k,l \in S$, $l \neq k$. Assuming interest rate zero, these quantities are directly linked to insurance contracts consisting of sojourn payments and payments upon transition.

\section{Forward-thinking and actuarial practice} \label{sec:applic}

Doubly-stochastic extension of classic actuarial multi-state models allows for the inclusion of systematic (undiversifiable) risk and market consistent valuation in accordance with the Solvency II regulatory framework, see e.g.\ the discussion in the beginning of \citet{Buchardt2014}. In itself, multi-state modeling gives rise to computational complications, which historically have been circumvented by imposing a suitable Markovian structure, whereby the transition probabilities can be found by solving ordinary differential equations. In the classic Markov chain life insurance setting, the jump process describing the state of the insured is assumed Markovian, and the computational task is reduced to solving the system of Kolmogorov forward equations. This is not the case when considering doubly-stochastic extensions, as any previous Markovian structure typically becomes void. In other words, the old weapons of the actuarial practitioner pose no threat to the new problems at hand. The development of mathematically sound definitions of forward transition rates is an attempt to once more stack the deck in favor of the actuarial practitioner. Conceptually, we are dealing with a whetstone for old weapons.

The practical relevance is essentially the following. The replacement conditions of \eqref{eq:kol_like}-\eqref{eq:replace_m00} allow for a two-step valuation procedure: First, calibrate the forward transition rates, and then calculate the cash flow using classic numerical schemes (solving systems of ordinary differential equations). If the first step is not too demanding, the actuarial practitioner can avoid implementing new advanced numerical schemes and instead rely on already available platforms. This approach can be a valuable shortcut to the implementation of systematic risks in the practitioner's current valuation software.

Two-step procedures are of course not necessary; a general alternative is to solve the system of Kolmogorov forward partial integro-differential equations, see e.g.\ \citet{Buchardt2017}. But the two-step approach also shows its strengths in a conceptual sense: it transforms the computational complications to a question of calibration of forward transition rates. It is our belief that this transformation is beneficial to e.g.\ actuarial practitioners searching for simple benchmark models. A similar way of thinking in a slightly different framework drives the work of \citet{ChristiansenNiemeyer}. We provide a substantiating example at the end of this subsection.

The concept of forward transition rates is derived from the concept of forward mortality, which again is inspired by the concept of forward interest rates. In the context of the latter and as an alternative to short-rate modeling, \citet{HeathJarrowMorton} propose a general framework, the so-called \textit{Heath-Jarrow-Morton framework}, where the modeling object of interest is the entire forward interest rate curve. In the context of longevity risk, a similar framework has been developed by \citet{Bauer2012}, where the marginal forward mortality curve rather than the stochastic mortality is the modeling object of interest. A similar change in modeling paradigm for multi-state settings might also prove valuable to practitioners. This requires mathematically sound definitions of forward transition rates (which we provide and discuss here) as well as the development of a framework similar to the Heath-Jarrow-Morton framework for doubly-stochastic multi-state Markov models (which we have postponed to future research).

In the following example, we illustrate the relevance of forward transition rates to actuarial practice as discussed above, both from a conceptual as well as a computational point of view.

\paragraph*{Survival model with surrender and free policy}

Consider the doubly-stochastic model illustrated in Figure~\ref{fig:key}. We assume that $\eta$, $\rho$, $\psi$, and $\sigma$ are non-negative and continuous, and that $\psi$ and $\sigma$ are also deterministic. Thus we allow for (possibly dependent) stochastic mortality and stochastic surrender rates.
\begin{figure}[h]
	\centering
	\scalebox{0.8}{
	\begin{tikzpicture}[node distance=2em and 0em]
		\node[punkt] (1) {dead};
		\node[anchor=north east, at=(1.north east)]{$2$};
		\node[punkt, left = 30mm of 1] (0) {active};
		\node[anchor=north east, at=(0.north east)]{$0$};
		\node[punkt, left = 30mm of 0] (2) {surrender};
		\node[anchor=north east, at=(2.north east)]{$3$};
		\node[punkt, below = 20mm of 1] (4) {dead \\ free policy};
		\node[anchor=north east, at=(4.north east)]{$4$};
		\node[punkt, left = 30mm of 4] (3) {free policy};
		\node[anchor=north east, at=(3.north east)]{$1$};
	\path
		(0)	edge [pil]		node [above]			{$\eta(\cdot)$}						(1)
		(0)	edge [pil]		node [above]			{$\rho(\cdot)$}					(2)
		(0)	edge [pil]		node [right]			{$\psi(\cdot)$}						(3)
		(3)	edge [pil]		node [below]			{$\eta(\cdot)$}						(4)
		(3)	edge [pil]		node [below left]		{$\rho(\cdot) + \sigma(\cdot)$} 	(2)
	;
	\end{tikzpicture}}
	\caption{Doubly-stochastic survival model with options of surrender and conversion to free policy and with stochastic mortality and stochastic baseline surrender rate.}
	\label{fig:key}
\end{figure}
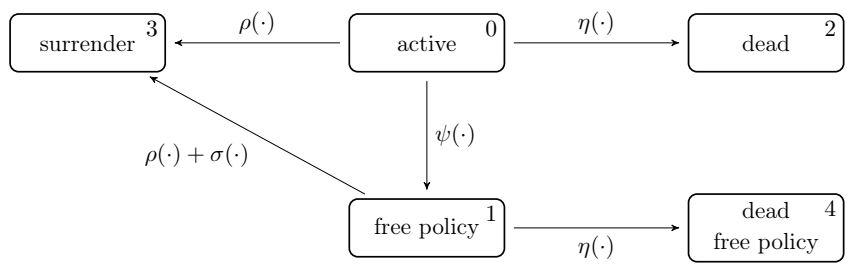
When $\eta$ and $\rho$ are deterministic and $\sigma=0$, we are within the class of models considered by \citet{BuchardtMoller2015}, see in particular Section 3.2 therein, where the connection to actuarial practice is also carefully explained.

%Fix $t\geq0$. %Following Subsection~\ref{sec:fer}, assume differentiability of auxiliary functions $\mathcal{P}_{jk}(t,\cdot)$ defined by \eqref{eq:uncond_prob}. Assuming we can interchange differentiation w.r.t.\ $T$ and integration w.r.t.\ $\mathbb{P}$,
Under certain regularity conditions, straightforward calculations (given in Appendix~\ref{ap:A}) show that the state-wise forward transition rates  given by \eqref{eq:buc_forward} take the form
\begin{align}
m_{01}(t,T)
&=
\psi(T), \nonumber \\
m_{02}(t,T)
=
m_{14}(t,T)
&=
\frac{
\mathbb{E}\!\left[e^{-\int_{(t,T]} \left(\eta(s) + \rho(s)\right) \, \mathrm{d}s} \eta(T) \, \big| \, \mathbb{F}_t^{\eta,\rho}\right]
}
{\mathbb{E}\!\left[e^{-\int_{(t,T]} \left(\eta(s) + \rho(s)\right) \, \mathrm{d}s} \, \big| \, \mathbb{F}_t^{\eta,\rho}\right]
}, \label{eq:eta}\\
m_{03}(t,T)
&=
\frac{
\mathbb{E}\!\left[e^{-\int_{(t,T]} \left(\eta(s) + \rho(s)\right) \, \mathrm{d}s} \rho(T) \, \big| \, \mathbb{F}_t^{\eta,\rho}\right]
}
{\mathbb{E}\!\left[e^{-\int_{(t,T]} \left(\eta(s) + \rho(s)\right) \, \mathrm{d}s} \, \big| \, \mathbb{F}_t^{\eta,\rho}\right]
}, \label{eq:rho} \\
m_{13}(t,T)
&=
m_{03}(t,T) + \sigma(T), \nonumber
\end{align}
with the state-wise forward transition rates being zero for the remaining indices.

The following two-step valuation procedure is now self-evident: First, calculate \eqref{eq:eta} and \eqref{eq:rho}, and then calculate cash flows using classic methods. To illustrate the possible advantages of the two-step procedure within this example, assume that $(\eta,\rho)$ belongs to the class of affine processes. Then \eqref{eq:eta} and \eqref{eq:rho} can be calculated by solving simple systems of ordinary differential equations, see e.g.\ \citet{Duffie2000}, \citet{Buchardt2016}, and \citet{LarsPhD} Chapter 5. In contrast, the general approach requires either solving the system of Kolmogorov forward partial integro-differential equations, see e.g.\ \citet{Buchardt2017}, or applying Monte Carlo methods. With reference to the study of numerical efficiency by \citet{Buchardt2016} in a comparable setting, we conclude that in the affine setting, the two-step procedure is more efficient than the general approach.

In this example, the advantage of the two-step approach is illustrated using the state-wise forward transition rates, however, the conclusion also holds for the forward equations rates. On the basis of \eqref{eq:eta} and \eqref{eq:rho} there exists an $\mathbb{F}_t^{\eta,\rho}$-measurable version of the state-wise forward transition rates. In particular, under certain regularity conditions, the forward equations rates and the state-wise forward transition rates must agree, which implies that the forward equations rates satisfy \eqref{eq:replace_m00}. Note also that if $\eta$ and $\rho$ are independent, we obtain exactly the marginal forward transitions rates.  It is the specific structure of the transition intensities within the model that makes the forward equations rates and the state-wise forward transition rates agree. Characterizing the class of models for which this is the case is postponed to future research

\section{Concluding remarks}\label{sec:remarks}

In the previous sections, we have focused solely on biometric and behavioral risks while not taking market risks and the time value of money into account. Hence we have only dealt with replacement arguments for the expected accumulated cash flow. In the context of reserving and pricing, interest lies in the prospective reserve, i.e.\ the expected present value of future payments. We now provide a short and informal discussion using market consistent valuation principles for life insurance and pensions, see e.g.\ \citet{SteffensenMoller2007}.

Let $r$ be some continuous short rate. If the short rate is deterministic, then the prospective reserve $V$ is simply given by
\begin{align*}
V(t)=\int_{(t,\infty)} \hspace{-3mm} e^{-\int_t^s r(u) \, \mathrm{d}u} \, A(t,\mathrm{d}s),
\end{align*}
assuming the integral exists. If \eqref{eq:replace_m0} and \eqref{eq:replace_m00} hold, then
\begin{align*}
V(t)
=
\int_{(t,\infty)} \hspace{-3mm} e^{-\int_t^s r(u) \, \mathrm{d}u} \sum_{k \in S} P_{X_tk}^m(t,s) \bigg(
b_k(s) + \hspace{-3mm} \sum_{l \in S, l \neq k} \hspace{-2mm} m_{kl}(t,s) b_{kl}(s)
\bigg) \, \mathrm{d}s,
\end{align*}
confer with \eqref{eq:Asotch2}. If the short rate is stochastic but the market risks are independent of the biometric and behavioral risks, the above instead reads
\begin{align} \label{eq:reserve}
V(t)=
\int_{(t,\infty)} \hspace{-3mm} e^{-\int_t^s f(t,u) \, \mathrm{d}u} \sum_{k \in S} P_{X_tk}^m(t,s) \bigg(
b_k(s) + \hspace{-3mm} \sum_{l \in S, l \neq k} \hspace{-2mm} m_{kl}(t,s) b_{kl}(s)
\bigg) \, \mathrm{d}s,
\end{align}
where $f(t,\cdot)$ is the usual forward interest rate associated with the short rate $r$. Thus as long as the markets risks are independent of the biometric and behavioural risks, the results and discussions of the previous sections extend from the expected accumulated cash flow to the prospective reserve in an immediate manner.

If there is dependency between the market risks and the biometric and behavioral risks, \eqref{eq:reserve} ceases to hold and the previous results and discussions are not directly extendable. Forward transition and interest rates in the context of dependency between markets risks and biometric and behavioral risks are therefore not discussed in this paper. To our knowledge, only \citet{Buchardt2014} has provided a forward rate concept allowing for successful replacement arguments in multi-state models with dependency between interest and transition rates. \citet{Buchardt2014} only considers simple models consisting of at most one non-absorbing state. A natural next step is to extend the definition of forward equations rates and the definition of state-wise forward transition rates to allow for dependency between market risks and biometric and behavioral risks and compare the concepts to that of \citet{Buchardt2014}.

\section*{Acknowledgements}

Christian Furrer's research is partly funded by the Innovation Fund Denmark (IFD) under File No.\ 7038-00007B. We would like to thank Lars Frederik Brandt for fruitful discussions.

\bibliography{references}
\bibliographystyle{plainnat}

\appendix

\section{Proofs} \label{ap:A}

\paragraph*{Proof of Lemma~\ref{lemma:cond}}

In the following we use standard arguments for conditional distributions, expectations and independence. The notation and methodology follows \cite{Beting}.
Because $X$ is Markovian conditionally on $\mu$, it holds that
\begin{align}\label{eq:ap0}
\mathbb{F}^X_{t+} \bigCI \mathbb{F}_t^X \mid \sigma(X_t)  \vee  \mathbb{F}^\mu_\infty.
\end{align}
Furthermore, by construction, the conditional distribution of $(X_s)_{s\leq t}$ given $\mathbb{F}^\mu_\infty$ is $\mathbb{F}_t^\mu$-measurable: it is only a function of $\mu$ through $(\mu_s)_{s \leq t}$, confer with the properties of the (conditional) transition probabilities $P_{jk}^\mu$. It follows from \cite{Beting} Theorem 2.1.5 that
\begin{align}\label{eq:ap1}
\mathbb{F}^\mu_\infty \bigCI \mathbb{F}_t^X \mid \mathbb{F}^\mu_t,
\end{align}
where we have employed the asymmetric formulation of conditional independence (see e.g.\ \cite{Beting} Theorem 3.3.7). In particular, using reduction (see e.g.\ \cite{Beting} Lemma 3.3.5)
\begin{align*}
\mathbb{F}^\mu_{t+} \bigCI \mathbb{F}_t^X \mid \mathbb{F}^\mu_t.
\end{align*}
Also, from \eqref{eq:ap0} and \cite{Beting} Theorem 3.4.1,
\begin{align}\label{eq:ap2}
\mathbb{F}^{X,\mu}_{t+} \bigCI \mathbb{F}_t^X &\mid \sigma(X_t)  \vee  \mathbb{F}^\mu_\infty, 
\end{align}
while from \eqref{eq:ap1} and \cite{Beting} Theorem 3.4.2,
\begin{align}\label{eq:ap3}
\mathbb{F}^\mu_\infty \bigCI \mathbb{F}_t^X &\mid \sigma(X_t)  \vee \mathbb{F}^\mu_t.
\end{align}
Combining \eqref{eq:ap2} and \eqref{eq:ap3} using the same argument as in \cite{Beting} Example 3.4.4,  we obtain
\begin{align*}
\mathbb{F}^{X,\mu}_{t+} \bigCI \mathbb{F}_t^X \mid \sigma(X_t)  \vee  \mathbb{F}^\mu_t
\end{align*}
as desired. \qed

\paragraph*{Proof of Theorem~\ref{thm:existsunique}}

We first show that there exists a unique solution to \eqref{eq:forward_eqs_def} for any $T \in (t,\infty)$.
Fix $T\in(t,\infty)$. In what follows we suppress $t$ notationally and write $\mathcal{P}'_{jk}(T)$ for $\frac{\partial}{\partial T}\mathcal{P}_{jk}(t,T)$. Because $\mathcal{P}_{jk}(T) = 0$ for $k < j$, it follows that $m_{jk}(T) = 0$ for $k < j$, and the remaining system of equations takes the form
\begin{align*}
A(T) \tilde{m}(T) = \tilde{\mathcal{P}}'(T),
\end{align*}
with $\tilde{m}$ being the vector $m$ after removing indexes $k < j$ and $\tilde{\mathcal{P}}'$ being the vector $\mathcal{P}'$ after removing the same indexes,
\begin{align*}
\tilde{m}&=\big(m_{01},\ldots,m_{0J},m_{12},\ldots,m_{1J},\ldots,m_{(J-1)J}\big)^\top\!, \\
\tilde{\mathcal{P}}'&=\big(\mathcal{P}'_{01},\ldots,\mathcal{P}'_{0J},\mathcal{P}'_{12},\ldots,\mathcal{P}'_{1J},\ldots,\mathcal{P}'_{(J-1)J}\big)^\top\!,
\end{align*}
and where $A$ is an upper triangular matrix with diagonal elements
\begin{align*}
\big(\overbrace{\mathcal{P}_{00},\ldots,\mathcal{P}_{00}}^{J\text{ times}},\overbrace{\mathcal{P}_{11},\ldots,\mathcal{P}_{11}}^{J-1\text{ times}},\ldots,\mathcal{P}_{(J-1)(J-1)}\big),
\end{align*}
and where the other entries are zeros or elements of $\mathcal{P}$.

Because
\begin{align*}
P_{jj}^\mu(T)
=
\exp\left\{-\int_t^T \sum_{k > j} \mu_{jk}(s) \, \mathrm{d}s \right\}
>
0,
\end{align*}
it holds that $\mathcal{P}_{jj}(T)>0$, hence in particular
\begin{align*}
\det A = \prod_{j=0}^{J-1} P_{jj}^{J-j} > 0,
\end{align*}
which implies that $A$ is invertible. Hence for fixed $T\in(t,\infty)$ there exists a unique solution given by $m_{jk}(T)=0$ for $k < j$ and
\begin{align*}
\tilde{m}(T) = A^{-1}(T)\tilde{\mathcal{P}}'(T).
\end{align*}
To complete the proof, we have to show that the solution is $\mathbb{F}_t^\mu$-measurable (as a function of $T$). This follows immediately by e.g.\ Cramer's rule if the entries of $A$ are $\mathbb{F}_t^\mu$-measurable. But the entries of $A$ are either zero or elements of $\mathcal{P}$, which are trivially $\mathbb{F}_t^\mu$-measurable, yielding the desired result.

If $\mathcal{P}'(t,\cdot)$ is assumed to be continuous, it follows by similar arguments and an application of e.g.\ Cramer's rule that the solution also is continuous. \qed

\paragraph*{Forward transition rates in the survival model with surrender and free policy}

We consider the state-wise forward transition rates given by \eqref{eq:buc_forward}. Because $\psi$ and $\sigma$ are deterministic, the only non-trivial derivations are related to $m_{13}$ and $m_{14}$. By setting $\sigma=0$ and using symmetry, the derivation of $m_{14}$ follows from the derivation of $m_{13}$. Hence it suffices to derive $m_{13}$. On $(X_t = 1)$ it holds that 
\begin{align*} 
&\mathbb{E}\!\left[ \left.\mathds{1}_{(X_T = 1)} \, \right| \sigma(X_t)  \vee  \mathbb{F}^{\eta,\rho}_t\right] \\ 
&= 
\mathbb{E}\!\left[ \left. e^{-\int_{(t,T]} (\eta(s) + \rho(s) + \sigma(s)) \, \mathrm{d}s}  \, \right| \mathbb{F}^{\eta,\rho}_t\right]\!, \\ 
&\mathbb{E}\!\left[ \left.\mathds{1}_{(X_T = 1)} \left(\rho(T)+\sigma(T)\right) \, \right| \sigma(X_t)  \vee  \mathbb{F}^{\eta,\rho}_t\right] \\ 
&= 
\mathbb{E}\!\left[ \left.e^{-\int_{(t,T]} (\eta(s) + \rho(s) + \sigma(s)) \, \mathrm{d}s} (\rho(T) + \sigma(T)) \, \right| \mathbb{F}^{\eta,\rho}_t\right]\!. 
\end{align*} 
Consequently, 
\begin{align*} 
m_{13}(t,T) 
&= 
\sigma(T) 
+ 
\frac{ 
\mathbb{E}\!\left[ \left.e^{-\int_{(t,T]} (\eta(s) + \rho(s)) \, \mathrm{d}s} \rho(T) \, \right| \mathbb{F}^{\eta,\rho}_t\right] 
} 
{ 
\mathbb{E}\!\left[ \left. e^{-\int_{(t,T]} (\eta(s) + \rho(s)) \, \mathrm{d}s}  \, \right| \,  \mathbb{F}^{\eta,\rho}_t\right] 
} 
\end{align*} 
on $(X_t = 1)$. Let now $C$ be defined by 
\begin{align*} 
C(t,T)
=
\int_{(t,T]} e^{-\int_{(t,s]} \psi(u) \, \mathrm{d}u} \psi(s) e^{-\int_{(s,T]} \sigma(u) \, \mathrm{d}u} \, \mathrm{d}s.
\end{align*}
Note that on $(X_t = 0)$,
\begin{align*}
&\mathbb{E}\!\left[ \left.\mathds{1}_{(X_T = 1)} \, \right| \sigma(X_t)  \vee  \mathbb{F}^{\eta,\rho}_t\right] \\
&=
\mathbb{E}\!\left[ \left. \int_{(t,T]} e^{-\int_{(t,s]} \left(\eta(u) + \rho(u) + \psi(u)\right) \, \mathrm{d}u} \psi(s) e^{-\int_{(s,T]} \left(\eta(u) + \rho(u) + \sigma(u)\right) \, \mathrm{d}u} \, \mathrm{d}s \, \right| \mathbb{F}^{\eta,\rho}_t\right] \\
&=
\mathbb{E}\!\left[ \left. 
e^{-\int_{(t,T]} \left(\eta(s) + \rho(s)\right) \, \mathrm{d}s}
\, \right| \mathbb{F}^{\eta,\rho}_t\right] C(t,T), \\
&\mathbb{E}\!\left[ \left.\mathds{1}_{(X_T = 1)}\left(\rho(T)+\sigma(T)\right) \, \right| \sigma(X_t)  \vee  \mathbb{F}^{\eta,\rho}_t\right] \\
&=
\mathbb{E}\!\left[ \left. 
e^{-\int_{(t,T]} \left(\eta(s) + \rho(s)\right) \, \mathrm{d}s} \left(\rho(T)+\sigma(T)\right)
\, \right| \mathbb{F}^{\eta,\rho}_t\right] C(t,T).
\end{align*}
Thus whenever $\psi$ is strictly positive on a subset of $(t,T]$ with non-zero Lebesgue measure, it holds on $(X_t = 0)$ that
\begin{align*}
m_{13}(t,T)
&=
\frac{
\mathbb{E}\!\left[ \left.e^{-\int_{(t,T]} \left(\eta(s) + \rho(s)\right) \, \mathrm{d}s} (\rho(T) + \sigma(T))
\, \right| \mathbb{F}^{\eta,\rho}_t\right]
}
{
\mathbb{E}\!\left[ \left. e^{-\int_{(t,T]} \left(\eta(s) + \rho(s)\right) \, \mathrm{d}s}
\, \right| \mathbb{F}^{\eta,\rho}_t\right]
} \\
&=
\sigma(T)
+
\frac{
\mathbb{E}\!\left[ \left.e^{-\int_{(t,T]} \left(\eta(s) + \rho(s)\right) \, \mathrm{d}s} \rho(T) \, \right| \mathbb{F}^{\eta,\rho}_t\right]
}
{
\mathbb{E}\!\left[ \left. e^{-\int_{(t,T]} \left(\eta(s) + \rho(s)\right) \, \mathrm{d}s}  \, \right| \,  \mathbb{F}^{\eta,\rho}_t\right]
},
\end{align*}
as the terms involving $C(t,T)$ cancel. To conclude, this shows that 
\begin{align*}
m_{13}(t,T)
=
\sigma(T)
+
\frac{
\mathbb{E}\!\left[ \left.e^{-\int_{(t,T]} \left(\eta(s) + \rho(s)\right) \, \mathrm{d}s} \rho(T) \, \right| \mathbb{F}^{\eta,\rho}_t\right]
}
{
\mathbb{E}\!\left[ \left. e^{-\int_{(t,T]} \left(\eta(s) + \rho(s)\right) \, \mathrm{d}s}  \, \right| \,  \mathbb{F}^{\eta,\rho}_t\right]
}
\end{align*}
is an $\mathbb{F}_t^{\eta,\rho}$-measurable version of the state-wise forward transition rates. \qed

\end{document}